\documentclass[10pt]{amsart}

\usepackage{amsmath, amsthm}
\usepackage{amssymb}
\usepackage{amscd}
\usepackage{latexsym}
\usepackage[latin1]{inputenc}
\usepackage[shortlabels]{enumitem}
\usepackage{extpfeil}

\usepackage{xypic}
\usepackage{graphicx}
\usepackage{changebar,color}

\usepackage{tikz-cd}

\newtheorem{theorem}{Theorem}[section]

\newtheorem{corollary}[theorem]{Corollary}
\newtheorem{lemma}[theorem]{Lemma}

\newtheorem{proposition}[theorem]{Proposition}

\newtheorem{definition}[theorem]{Definition}

\theoremstyle{remark}

\theoremstyle{definition}

\theoremstyle{definition}

\numberwithin{equation}{section}

\usetikzlibrary{decorations.markings}
\makeatletter
\tikzcdset{
open/.code={\tikzcdset{hook, circled};},
closed/.code={\tikzcdset{hook, slashed};},
open'/.code={\tikzcdset{hook', circled};},
closed'/.code={\tikzcdset{hook', slashed};},
circled/.code={\tikzcdset{markwith={\draw (0,0) circle (.375ex);}};},
slashed/.code={\tikzcdset{markwith={\draw[-] (-.4ex,-.4ex) -- (.4ex,.4ex);}};},
markwith/.code={
\pgfutil@ifundefined{tikz@library@decorations.markings@loaded}%
{\pgfutil@packageerror{tikz-cd}{You need to say %
\string\usetikzlibrary{decorations.markings} to use arrow with markings}{}}{}%
\pgfkeysalso{/tikz/postaction={/tikz/decorate,
/tikz/decoration={
markings,
mark = at position 0.5 with
{#1}}}}},
}
\makeatother

\renewcommand{\phi}{\varphi}

\def\A{\mathbb{A}}

\def\N{\mathbb{N}}
\def\Z{\mathbb{Z}}
\def\PP{\mathbb{P}}
\def\Q{\mathbb{Q}}

\newcommand{\cO}{{\mathcal O}}

\newcommand{{\OL}}{{{\mathcal O}_L}}

\newcommand{\lra}{\longrightarrow}

\newcommand{\lct}{{\mathrm{lct}}}

\renewcommand{\epsilon}{\varepsilon}

\newcommand{\SL}{\mathrm{SL}}
\newcommand{\GL}{\mathrm{GL}}

\title{A note on the semistability of singular projective hypersurfaces}

\author{Thomas Mordant}

\address{Universit\'e Paris-Saclay, Laboratoire de Math\'ematiques d'Orsay, 91405 Orsay Cedex, France}
\email{thomas.mordant@universite-paris-saclay.fr}

\begin{document}

\maketitle

\begin{abstract}
In this note, we give sufficient conditions for the (semi)stability of a  hypersurface  $H$ of $\mathbb{P}^N_k$ in terms of its degree $d$, the maximal multiplicity $\delta$ of its singularities, and the dimension $s$ of its singular locus.  For instance, we show that $H$ is semistable when $d \geq \delta \min (N+1, s+3)$. The proof relies in particular on Benoist's lower bound 
for the dimension of the intersection of the singular locus $H_{\mathrm{sing}}$ of $H$ with some linear subspace of $\mathbb{P}^N_k$ associated to a one-parameter subgroup $\lambda$ of $\mathrm{SL}_{N+1, k}$, in terms of the numerical data in the Hilbert-Mumford criterion applied to $\lambda$ and to an equation $F_H$ of~$H$.

\end{abstract}

\tableofcontents

\section{Introduction}

\subsection{} Throughout this note, we shall fix two integers $N \geq 1$ and $d \geq 2$, an algebraically closed field~$k$, and a hypersurface $H$ of degree $d$ in the projective space $\PP^N_k$, that is an effective Cartier divisor defined by the vanishing of a homogeneous polynomial $F_H \in k[X_0,\dots,X_N]_d \setminus \{0\}$.

We are interested in the (semi)stability of the hypersurface $H$, namely in the (semi)stability in the sense of geometric invariant theory, of the point $[F_H] \in \PP^{\binom{N+d}{d}-1}(k)$ with respect to the natural action of $\SL_{N+1,k}$ on $k[X_0,\dots,X_N]_d$.

The purpose of this note is to prove the following      
theorem.

\begin{theorem}\label{MainThm} Let $\delta$ be the maximal multiplicity of $H$ at a point of $H(k)$, and let $s$ be the dimension of the singular locus $H_{\mathrm{sing}}$ of $H$.

\begin{enumerate}
\item \label{HM isol mult} If  the following condition holds:
\begin{equation} \label{eq HM isol mult}
d \geq \delta \, \min(N+1, s+3)  \quad (\mbox{resp. }  d > \delta \, \min(N+1, s+3)\, ),
\end{equation}
then $H$ is semistable (resp. stable).

\item  \label{HM isol not cone} Assume $N\geq 2.$ If for every point $P \in H(k)$ where $H$ has multiplicity $\delta$,  
the projective tangent cone $\PP(C_P H)$ in $\PP(T_P \PP^N_k) \simeq \PP^{N-1}_k$ is not the cone\footnote{When $N=2$, this condition has to be interpreted as follows: \emph{the support of $\PP(C_P H)$ is not a unique point in $\PP(T_P\PP^2_k) \simeq \PP_k^{1}$}.} over  some hypersurface in a projective hyperplane of $\PP_k^{N-1}$, and  if  the following condition holds:
\begin{equation} \label{eq HM isol not cone}
d \geq (\delta-1) \min(N+1, s+3)  \quad (\mbox{resp. }  d > (\delta-1)  \min(N+1, s+3)\, ),
\end{equation} 
then $H$ is semistable (resp. stable).
\end{enumerate}

\end{theorem}

The integer $\delta$ may take any value in $\{1, \dots, d\},$ and the integer $s$ any value in $\{-1, 0, \dots, N-1\}$.  Indeed the divisor   $H$ is smooth, or equivalently $H_{\mathrm{sing}}$ is empty, iff $\delta= 1$, or iff $s=-1$.  Moreover $H$ is reduced iff $s \leq N-2$, and (when $N \geq 2$) is normal iff $s \leq N-3$.

Observe that the condition on the projective tangent cones in (2) may hold only when $\delta \geq 2$.

\subsection{} Let us spell out a few special cases of Theorem \ref{MainThm}.

\subsubsection{}\label{MainThmSmooth}

Applied with $\delta =1$ and $s=-1$, Part (1) shows that, when the hypersurface $H$ is smooth, it is semistable (resp. stable) if $d\geq 2$ (resp. if $d>2$). 

This is a classical result of Mumford \cite[Proposition 4.2]{MumfordFogartyKirwan94}: the semistability is a consequence of the characterization of the smoothness of $H$ by the non-vanishing of the discriminant of the homogeneous polynomial $F_H$; the stability when $d >2$ follows from the finiteness, in this case, of the group of projective transformations of $\PP_k^N$ which leave $H$ invariant.\footnote{In \emph{loc. cit.}, Mumford analyses the characteristic zero situation, where this follows from the vanishing of $\Gamma(H, T_H)$, established by Kodaira and Spencer \cite[Lemma 14.2]{Kodaira-Spencer58}. This finiteness in characteristic zero is actually a classical result of Jordan \cite{Jordan80}, which has been extended in positive characteristic by Matsumura and Monsky; see  \cite[Theorem 1]{Matsumura-Monsky64}. For a discussion of the construction and of the properties of the discriminant of homogeneous polynomials in any number of variables, the reader may refer to \cite{GKZ08}, and, for a presentation also valid in positive characteristic, to \cite{Demazure12}.}

\subsubsection{} Applied with $s=0,$ Part (1) shows that, when the hypersurface $H$ has only isolated singularities, it is semistable (resp. stable) when one of the following conditions is satisfied:
\begin{equation}\label{N1isolated}
 N= 1 \quad \mbox{and} \quad d \geq 2 \delta \quad (\mbox{resp. } d > 2 \delta),
 \end{equation}
or:
\begin{equation}\label{N>1isolated} N\geq  2 \quad \mbox{and} \quad d \geq 3 \delta \quad (\mbox{resp. } d > 3 \delta).
\end{equation}

In case \eqref{N1isolated}, this criterion for (semi)stability  is indeed a classical result, established in \cite[Chapter 4, \S 1]{MumfordFogartyKirwan94} as a direct consequence of the Hilbert-Mumford criterion\footnote{The Hilbert-Mumford criterion also shows that, when $N=1$, these conditions are necessary for the (semi)stability.}; see also \cite[10.2]{Dolgachev03} and \cite[Prop. 7.9]{Mukai03}. 

Applied to non-smooth hypersurfaces, the criterion in \eqref{N>1isolated} is  significant  when $d \geq 6$ (resp. when $d > 6$), in which case it appears to be new.  

\subsubsection{}  We shall use the semistability criteria in Theorem \ref{MainThm} to compare the Griffiths-Kato height of a pencil of projective hypersurfaces, which we have computed in \cite{Mordant22}, with 
its GIT height, defined by using the associated curve in the GIT quotient $\PP(k[X_0,\dots, X_N]_d)/\!/ \SL_{N+1,k};$ see \cite{MordantGIT}.  

To achieve this, we will apply Theorem \ref{MainThm} to hypersurfaces admitting only isolated semi-homo\-gene\-ous singularities  (aka ordinary multiple points), namely singular points $P$ such that the projective tangent cone $\PP(C_P H)$ is a smooth hypersurface in  $\PP(T_P \PP^N_k).$  When this holds and when $N\geq 2$ and $\delta \geq 2$, Part (2) of Theorem \ref{MainThm} establishes that the hypersurface $H$ is semistable (resp. stable) if the following condition is satisfied:
\begin{equation}
d \geq 3 (\delta -1) \quad (\mbox{resp. } d > 3 (\delta -1)) .
\end{equation}

\subsubsection{}  Our main tool in order to establish Theorem \ref{MainThm} is a remarkable estimate due to Benoist \cite[Lemma 3.2]{Benoist14}, which provides a lower bound for the dimension of the intersection of the singular locus $H_{\mathrm{sing}}$ of $H$ with some linear subspace of $\PP^N_k$ associated to some one-parameter subgroup $\lambda$ of $\SL_{N+1, k}$ in terms of the numerical data in the Hilbert-Mumford criterion applied to $\lambda$ and $F_H$. 

This lower bound is used in \cite{Benoist14} to establish the stability of Hilbert points associated to large classes of smooth complete intersections in projective spaces. The arguments in \cite{Benoist14} provide a proof of the (semi)stability of smooth hypersurfaces discussed in \ref{MainThmSmooth}  above that avoids the use of the discriminant and of finiteness theorems \emph{\`a la} Jordan in the classical  proof.\footnote{The case of smooth hypersurfaces is covered by the general results  in \cite{Benoist14}, but in this case, the proof of (semi)stability becomes much simpler: it is  a straightforward consequence of \cite[Lemma 3.5 and 3.6]{Benoist14} together with the Hilbert-Mumford criterion.}

It turns out that Benoist's lower bound also leads to (semi)stability criteria valid for singular hypersurfaces, as demonstrated by Theorem \ref{MainThm}.  Its proof will be given in Section \ref{PfMTvar} below, after we have formulated the Hilbert-Mumford criterion for the (semi)stability of hypersurfaces and established some preliminary estimates in Section \ref{secHM}, and then recalled Benoist's lower bound and stated some relevant corollaries in Section  \ref{secBenoist}. 
We also illustrate the flexibility of our method by discussing some variants of Theorem \ref{MainThm} in Subsection \ref{subsecVar}.

Geometric invariant theory is presented in several beautiful introductory texts (\cite{Kraft84}, \cite{Dolgachev03}, \cite{Mukai03}, \cite{Hoskins16}) which are accessible with a limited background in algebraic geometry. Our proof of Theorem \ref{MainThm} based on Benoist's lower bound relies only on a few basic results in algebraic geometry, and should be  accessible to a beginner in algebraic geometry, with some familiarity with the basic results of geometric invariant theory as presented in these texts.

\subsection{} To the best of our knowledge, the only results concerning the (semi)stability of singular hypersurfaces  of arbitrary large degree or dimension  available  in the literature appear in the works of  Tian \cite{Tian94} and Lee \cite{Lee08}, valid in characteristic zero. Lee's results have been extended by Okawa \cite{Okawa11} to positive characteristic.

In an appendix (Section \ref{LeeLC}), we describe the results of Lee and Tian, and compare them with ours. 
An advantage of our approach compared with  Tian and Lee's ones is that these   require the condition:
$$d \geq N+1.$$
Another point in favor of our approach is that it is directly valid in arbitrary characteristic.

For special values of $N\geq 2$ and $d\geq 2$, various classification results,\footnote{valid in characteristic zero, and in large enough characteristic.} established by means of the Hilbert-Mumford criterion, provide characterizations of  (semi)stable hypersurfaces of degree $d$ in $\PP^N$. Notably the following results are classical:

\begin{enumerate}

\item When $d = 2$, the quadric $H$ is semistable if and only if it is smooth, and it is never stable; see for instance \cite[Chapter 4, \S 2]{MumfordFogartyKirwan94}, 
and  \cite[Example 10.1]{Dolgachev03}.

\item When $(N,d) = (2, 3)$, the cubic curve $H$ is semistable (resp. stable) if and only if it admits only ordinary double points (resp. it is smooth); this goes back to Hilbert's seminal paper \cite{Hilbert93}, where a basic version of the Hilbert-Mumford appears; see  \cite[10.3]{Dolgachev03} for a modern treatment.

\item When $(N,d) = (3,3)$, the cubic surface $H$ is semistable (resp. stable) if and only if it admits only ordinary double points or cusps\footnote{A cusp is a singularity where the equation of the hypersurface admits a local expansion of the form
$
x^2 + y^2 + z^3 = 0.
$
} (resp. only ordinary double points); see \cite{Hilbert93}, 
and  \cite[Proposition 6.5]{Beauville09} for a detailed exposition.
\end{enumerate}

Other complete descriptions of semistable and stable hypersurfaces have been found for $(N,d) = (3,4)$ (quartic surfaces) by Shah (see for instance \cite{Shah81}), for $(N,d) = (4,3)$ (cubic three-folds) by Allcock  \cite{Allcock03} and by Yokoyama \cite{Yokoyama02},
and for $(N,d) = (5, 3)$ (cubic four-folds) by Laza  \cite{Laza09} and by Yokoyama \cite{Yokoyama08}. See also \cite[IV.4.2, V.5.2, and VI.6.7]{Huybrechts23} for a description of these results in the general context of the geometry of  cubic hypersurfaces. 

The complexity of these results indicates that one cannot expect to establish complete classification results for general values of $(N,d)$. 

\subsection{Acknowledgments} I thank Jean-Beno\^it Bost and Gerard Freixas for their suggestions and comments concerning earlier versions of this note, and Damien Simon for his help with its typing. I owe an intellectual debt to Olivier Benoist, whose beautiful paper \cite{Benoist14} provided the inspiration for this work. I am grateful to an anonymous referee for a careful reading of this note, and helpful comments and suggestions for its improvement.

During the work leading to this note, I benefitted from the support of the Fondation Math\'ematique Jacques Hadamard, as a \emph{Hadamard Lecturer.}

\section{The Hilbert-Mumford criterion for projective hypersurfaces}\label{secHM}

\subsection{$\alpha$-degree and the Hilbert-Mumford criterion.}
To establish the (semi)stability of hypersurfaces of degree $d$ in $\PP^N_k,$ we will apply the Hilbert-Mumford criterion (see \cite{Hilbert93}, \cite[Theorem 2.1]{MumfordFogartyKirwan94}) to the natural action of $\SL_{N+1,k}$ on $k[X_0,\dots,X_N]_d$. 

To achieve this, it is convenient to introduce the notion of  \emph{$\alpha$-degree} of a nonzero homogeneous polynomial in $k[X_0, \dots, X_N]_d$, as in \cite[2.3]{Benoist14}.

\begin{definition}
    \label{alpha degree}
Let $\alpha = (\alpha_0,\dots,\alpha_N)$ be an element of $\Z^{N+1}$ and $m=(m_0, \dots, m_N)$ an element of $\N^{N+1}$. The $\alpha$-degree of the monomial $X^m:= X_0^{m_0} \dots X_N^{m_N}$ is the integer defined by:
$$
\deg_{\alpha} X^m := \sum_{i = 0}^N m_i \alpha_i.
$$
The $\alpha$-degree $\deg_{\alpha} F$ of a nonzero homogeneous polynomial
$$
F(X_0,\dots,X_N) :=
\sum_{\substack{m \in \N^{N+1}\\ \vert m \vert = d}} \lambda_m X^m \in k[X_0,\dots,X_N]_d \setminus \{0\}
$$
is defined to be the maximum of the $\alpha$-degrees of the monomials appearing in $F$ with nonzero coefficients:
$$
\deg_{\alpha} F  := \underset{\substack{m \in \N^{N+1} \\ \vert m \vert = d \\ \lambda_m \neq 0}}{\mathrm{max}} \deg_{\alpha}X^m.
$$
\end{definition}

Indeed, with this notation, it is straightforward that the Hilbert-Mumford criterion takes the following form (see for instance \cite[Lemme 2.13]{Benoist14}): 

\begin{proposition}
\label{degree and HM}
The hypersurface $H$ is semistable (resp. stable) if and only if, for every $g \in \GL_{N+1}(k)$ and for every multi-index $\alpha = (\alpha_0,\dots,\alpha_N) \in \Z^{N+1} \setminus \{0\}$, the following implication holds:
\begin{equation}\label{eqHM}
\sum_{i = 0}^N \alpha_i = 0 \Longrightarrow \deg_{\alpha}(F_H \circ g) \geq 0 \quad(\mbox{resp. } \deg_{\alpha}(F_H \circ g) > 0).
\end{equation}
\end{proposition}

Observe that, after possibly changing $g$, it is enough to check this criterion on the multi-indices $(\alpha_0,\dots,\alpha_N)$ such that $\alpha_0 \leq \dots \leq \alpha_N$.

\subsection{Some lower bounds on $\deg_\alpha (F_H \circ g)$}

The following proposition establishes some preliminary lower bounds on the $\alpha$-degree $\deg_\alpha (F_H \circ g)$ which occurs in the Hilbert-Mumford criterion  \eqref{eqHM} in terms of the geometry of the hypersurface $H$ at the point: $$P:=g ([0:\dots:0:1]).$$

\begin{proposition}
\label{alpha deg loc}
Let $g$ be an element of $\GL_{N+1}(k)$ and $P:=g ([0:\dots:0:1])$,  and let $\alpha = (\alpha_0,\dots,\alpha_N)$ be a multi-index in $\Z^{N+1} \setminus \{0\}$ such that $\alpha_0 \leq \dots \leq \alpha_N$ and $\sum_{i = 0}^N \alpha_i = 0$. 

\begin{enumerate}
\item \label{alpha deg loc non H} If $P$ is not in the hypersurface $H$, then the following inequality holds:
\begin{equation}
\label{ineq alpha deg loc non H} \deg_{\alpha}(F_H \circ g) \geq d\, \alpha_N > 0.
\end{equation}

\end{enumerate}

In the following statements, we assume that $P$ is in $H$ and we denote by $\delta_P$ the multiplicity of $H$ at $P$.

\begin{enumerate}[resume]
\item \label{alpha deg loc sing mult} The following inequality holds:
\begin{equation}
\label{ineq alpha deg loc sing mult} \deg_{\alpha}(F_H \circ g) \geq (d-2 \delta_P) \alpha_N - \delta_P \sum_{i = 1}^{N-1} \alpha_i.
\end{equation}
\end{enumerate}

From now on, we also assume that $N \geq 2.$

\begin{enumerate}[resume]
\item \label{alpha deg loc sing non plane} If the support of the projective tangent cone $\PP(C_P H) \subset \PP(T_P \PP^N_k) \simeq \PP^{N-1}_k$ to the hypersurface~$H$ at $P$ is not a projective hyperplane in $\PP(T_P \PP^N_k)$, then the following inequality holds:
\begin{equation}
\label{ineq alpha deg loc sing non plane} \deg_{\alpha}(F_H \circ g) \geq (d - 2 \delta_P + 1) \alpha_N - (\delta_P - 2) \alpha_1 - (\delta_P - 1) \sum_{i = 2}^{N-1} \alpha_i.
\end{equation}

\item \label{alpha deg loc sing non cone} If the projective tangent cone $\PP(C_P H)$ 
 is not the cone\footnote{As in Theorem \ref{MainThm}, when $N=2$, this condition has to be interpreted as follows: \emph{the support of $\PP(C_P H)$ is not a unique point in $\PP(T_P\PP^2_k) \simeq \PP_k^{1}$}.} over  some hypersurface in a projective hyperplane of $\PP(T_P \PP^N_k)$, then the following inequality holds:
\begin{equation}
\label{ineq alpha deg loc sing non cone} \deg_{\alpha}(F_H \circ g) \geq (d - 2 \delta_P + 1) \alpha_N - (\delta_P - 1) \sum_{i = 1}^{N-2} \alpha_i - (\delta_P - 2) \alpha_{N-1}.
\end{equation}
\end{enumerate}
\end{proposition}

\begin{proof}
Let $(x_0,\dots,x_{N-1})$ be the coordinate system induced by the homogeneous coordinates $X_0,\dots,$ $X_N$ on the affine open subset $(X_N \neq 0)$, and let $f(x_0,\dots,x_{N-1})$ be the polynomial defined by the restriction of the homogeneous polynomial $F_H \circ g$ to this subset. Observe that the intersection subscheme $g^{-1}(H) \cap (X_N \neq 0)$ is precisely the zero locus of $f$.

For \eqref{alpha deg loc non H}, if $P$ is not a point in the hypersurface $H$, i.e. if the point $[0 : \dots : 0 : 1]$ satisfies that the value $(F_H \circ g)([0 : \dots : 0 : 1])$ does not vanish, then the polynomial $f$ has a nonzero constant term, i.e. the polynomial $F_H \circ g$ admits a monomial in $X_N^d$ with nonzero coefficient. By definition of the $\alpha$-degree, the following inequality holds:
$$
\deg_{\alpha}(F_H \circ g) \geq d\,  \alpha_N > 0,
$$
where we have used that the vector $\alpha$ is nonzero and that $\alpha_N$ is its highest entry, so it is positive. This shows \eqref{alpha deg loc non H}.

Now let us assume that $P$ is 
in $H$, and let us denote by $\delta_P$ the multiplicity of $H$ at $P$. Observe that the multiplicity of the hypersurface $g^{-1}(H) = (F_H \circ g = 0)$ at the 
point $g^{-1}(P) = [0 : \dots : 0 : 1]$ is also $\delta_P$. By definition of the multiplicity, this means that every monomial appearing in $f$ with nonzero coefficient has degree at least $\delta_P$, and that the homogeneous part of degree $\delta_P$ of $f$, which we shall denote by $f_{\delta_P}$, does not vanish.

For \eqref{alpha deg loc sing mult}, we simply choose a monomial appearing with nonzero coefficient in $f_{\delta_P}$; it is of the form $x_{j_1} \dots x_{j_{\delta_P}}$ where $j_1,\dots,j_{\delta_P}$ are $\delta_P$ integers between $0$ and $N-1$. This means that the polynomial $F_H \circ g$ admits a monomial of the form $X_{j_1} \dots X_{j_{\delta_P}} X_N^{d-\delta_P}$ with nonzero coefficient. Consequently 
the following inequality holds:
$$
\deg_{\alpha}(F_H \circ g) \geq \sum_{p = 1}^{\delta_P} \alpha_{j_p} + (d - \delta_P) \alpha_N
\geq \delta_P \, \alpha_0 + (d - \delta_P) \alpha_N
= (d-2 \delta_P) \alpha_N - \delta_P \sum_{i = 1}^{N-1} \alpha_i,
$$
where we have first used the inequality $\alpha_0 \leq \alpha_j$, then  the equality $\sum_{i = 0}^N \alpha_i = 0$. 
This shows \eqref{alpha deg loc sing mult}.

For the remainder of the proof, we assume that $N \geq 2$ and observe that under the following isomorphism of projective spaces:
$$
\PP^{N-1}_k  \simeq \PP(T_0 \A^N_k)  
\lra \PP(T_{[0 : \dots : 0 : 1]} \PP^N_k)
\overset{\mathrm{d}g}{\lra} \PP(T_P \PP^N_k),
$$
the projective tangent cone $\PP(C_P H) \subset \PP(T_P \PP^N_k)$ is identified with the hypersurface of degree $\delta_P$ in~$\PP^{N-1}_k$ defined by the vanishing of the homogeneous polynomial $f_{\delta_P}(x_0,\dots,x_{N-1})$.

For \eqref{alpha deg loc sing non plane}, if the support of the projective tangent cone $\PP(C_P H) \subset \PP(T_P \PP^N_k)$ is not a projective hyperplane, then using the above remark, the polynomial $f_{\delta_P}$ is not a multiple of $x_0^{\delta_P}$. Consequently this polynomial admits a monomial with nonzero coefficient of the form $x_{j_1} \dots x_{j_{\delta_P}}$ where $j_1,\dots,j_{\delta_P}$ are $\delta_P$ integers that are not all equal to $0$. Without loss of generality, we can assume that $j_{\delta_P}$ is at least $1$. 

Then the polynomial $F_H \circ g$ admits a monomial of the form $X_{j_1} \dots X_{j_{\delta_P}} X_N^{d-\delta_P}$ with nonzero coefficient. Consequently the following inequalities hold:
\begin{align*}
\deg_{\alpha}(F_H \circ g) &\geq \sum_{p = 1}^{\delta_P} \alpha_{j_p} + (d - \delta_P) \alpha_N\\
&\geq \sum_{p = 1}^{\delta_P - 1} \alpha_{j_p} + \alpha_{j_{\delta_P}} + (d - \delta_P) \alpha_N\\
&\geq (\delta_P-1) \alpha_0 + \alpha_1 + (d - \delta_P) \alpha_N\\
&\geq (d - 2 \delta_P + 1) \alpha_N - (\delta_P - 2) \alpha_1 - (\delta_P - 1) \sum_{i = 2}^{N-1} \alpha_i.
\end{align*}
This shows \eqref{alpha deg loc sing non plane}.

Finally let us show \eqref{alpha deg loc sing non cone}. 
If $N \geq 3$ and the projective tangent cone $\PP(C_P H) \subset \PP(T_P \PP^N_k) \simeq \PP^{N-1}_k$ is not the cone over some hypersurface in a projective hyperplane of $\PP_k^{N-1}$, then using the above remark, the polynomial $f_{\delta_P}(x_0,\dots,x_{N-1})$ actually depends on $x_{N-1}$: it is not a polynomial in $x_0,\dots,x_{N-2}$. The same holds if $N = 2$ and the support of $\PP(C_P H)$ is not a unique point in $\PP(T_P\PP^N_k) \simeq \PP_k^{1}$. In both cases, the polynomial $f_{\delta_P}$ admits a monomial with nonzero coefficient of the form $x_{j_1} \dots x_{j_{\delta_P-1}} x_{N-1}$ where $j_1,\dots,j_{\delta_P-1}$ are $\delta_P-1$ integers between $0$ and $N-1$. 

Consequently the polynomial $F_H \circ g$ admits a monomial of the form $X_{j_1} \dots X_{j_{\delta_P-1}} X_{N-1} X_N^{d-\delta_P}$ with nonzero coefficient, and the following inequalities hold:
\begin{align*}
\deg_{\alpha}(F_H \circ g) &\geq \sum_{p = 1}^{\delta_P-1} \alpha_{j_p} + \alpha_{N-1} + (d - \delta_P) \alpha_N\\
&\geq (\delta_P-1) \alpha_0 + \alpha_{N-1} + (d - \delta_P) \alpha_N\\
&\geq (d - 2 \delta_P + 1) \alpha_N - (\delta_P - 1) \sum_{i = 1}^{N-2} \alpha_i - (\delta_P - 2) \alpha_{N-1}.
\end{align*}
This shows \eqref{alpha deg loc sing non cone} and concludes the proof.
\end{proof}

\section{$\alpha$-degree and dimension of the singular locus:  Benoist's lower bound}\label{secBenoist}

\subsection{} The following result, due to Benoist \cite[Lemme 3.2]{Benoist14}, relates the $\alpha$-degree $\deg_\alpha (F_H \circ g)$ in the Hilbert-Mumford criterion \eqref{eqHM} and the dimension of the singular locus $H_{\mathrm{sing}}$  of $H$. 

\begin{proposition}
\label{alpha sing}
Let $g$ be an element of $\GL_{N+1}(k)$, $\alpha = (\alpha_0,\dots,\alpha_N)$ a multi-index in $\Z^{N+1} \setminus \{0\}$ such that: $$\alpha_0 \leq \dots \leq \alpha_N$$ and: $$\sum_{i = 0}^N \alpha_i = 0,$$
 and $(u,v, s)$ an element of $\N^3$ such that: $$u+v+s = N.$$
 
  If the following inequality holds:
\begin{equation}
\label{ineq cond alpha sing}
\deg_{\alpha}(F_H \circ g) < \alpha_u + (d-1)\alpha_v, 
\end{equation}
then the dimension of the singular locus $H_{\mathrm{sing}}$  of $H$ satisfies the following lower bound:
\begin{equation}
\label{ineq dim sing}
\dim(H_{\mathrm{sing}}) \geq \dim\big(H_{\mathrm{sing}} \cap ((g \cdot X )_0 = \dots = (g \cdot X )_{v-1} = 0)\big) \geq s,
\end{equation}
where $(g\cdot X)_i$ denotes the linear form $\sum_{j=0}^N g_{ij}X_j.$ 
\end{proposition}

In \eqref{ineq dim sing}, 
when the scheme $H_{\mathrm{sing}}$ or $H_{\mathrm{sing}} \cap ((g \cdot X )_0 = \dots = (g \cdot X )_{v-1} = 0)$ is empty, its dimension is understood  to be  $-1.$ Consequently when this is the case, condition \eqref{ineq cond alpha sing} cannot be satisfied.

For the convenience of the reader, we provide some details on the proof of Proposition \ref{alpha sing}.

\begin{proof}
Replacing $H$ by $g^{-1}(H)$, we can assume that $g$ is the identity matrix.

The polynomial $F_H$ admits a unique decomposition of the form:
$$
F_H = X_0 P_0 + \dots + X_N P_N,
$$
where for every integer $i$ such that $0 \leq i \leq N$, $P_i$ is a homogeneous polynomial of degree $d-1$ in the indeterminates $X_i,\dots,X_N$.

The remainder of the proof of Proposition \ref{alpha sing} shall rely on the following lemmas.

\begin{lemma}
\label{alpha sing pol depends} Assume that the following inequality holds:
\begin{equation}\label{ineq cond alpha sing bis} \deg_{\alpha} F_H  < \alpha_u + (d-1)\alpha_v.
\end{equation}
Then for every integer $i \in \{u, \dots, N\}$, every monomial appearing in $P_i$ with a nonzero coefficient is divisible by some $X_j$ with $0\leq j < v$. 
\end{lemma}

\begin{proof}[Proof of Lemma \ref{alpha sing pol depends}]
Let $i$ be an integer in  $\{u, \dots, N\}$, and let $X_0^{m_0} \dots X_N^{m_N}$ be a monomial appearing with nonzero coefficient in the homogeneous polynomial $P_i$, where $m_0,\dots,m_N$ are $N+1$ non-negative integers such that  $\sum_{j = 0}^N m_j =d-1$.

 Since the monomial $X_i \cdot (X_0^{m_0} \dots X_N^{m_N})$ appears in $F_H$ with a nonzero coefficient, the definition of the $\alpha$-degree and inequality \eqref{ineq cond alpha sing bis} imply the following inequalities:
$$
\alpha_i + \sum_{j = 0}^N m_j \alpha_j \leq \deg_{\alpha} F_H  < \alpha_u + (d-1) \alpha_v.
$$
Moreover, since $i \geq u$, the coefficient $\alpha_i$ is at least $\alpha_u$, and therefore  the following inequality holds:
\begin{equation}\label{sum m v}
\sum_{j = 0}^N m_j \alpha_j < (d-1) \alpha_v.
\end{equation}

Together with the implication:
$$j \geq v \Longrightarrow \alpha_j \geq \alpha_v$$
and the equality:
$$\sum_{j = 0}^N m_j =d-1,$$
the inequality \eqref{sum m v} implies the existence of an integer $j$ such that $0\leq j < v$ and that $m_j$ is nonzero. 
Then the monomial $X_0^{m_0} \dots X_N^{m_N}$ is divisible by $X_j$, as required.
\end{proof}

\begin{lemma}\label{Z sing}
Let us moreover assume that the inequality $u \leq v$
is satisfied.  The closed subscheme $Z$ in $\PP^N_k$ defined by the following $u+v$ equations:
$$
X_0 = \dots = X_{v-1} = P_0 = \dots = P_{u-1} = 0
$$
 is contained in  $H_{\mathrm{sing}} \cap (X_0 = \dots = X_{v-1} = 0)$.
 \end{lemma}

\begin{proof}[Proof of Lemma \ref{Z sing}]
Since, using Lemma \ref{alpha sing pol depends}, for every integer $i \in \{u, \dots, N\}$, every monomial with nonzero coefficient in $P_i$ is divisible by some $X_j$ with $0 \leq j < v$, and since every such $X_j$ vanishes on $Z$ by definition, we obtain that for every $i \in \{u, \dots, N\}$, the polynomial $P_i$ vanishes on $Z$.

Consequently, for every integer $i \in \{0, \dots, N \},$ the polynomial $P_i$ vanishes on $Z$, and therefore the sum: 
$$
F_H = \sum_{i = 0}^N X_i P_i
$$
vanishes on $Z$. This shows that $F_H$ vanishes on $Z$ and therefore that $Z$ is contained in $H$.

Now let us fix an integer $i_0$ such that $0 \leq i_0 \leq N$ and show that the partial derivative:
\begin{equation}
\label{decom partial Z} 
\partial_{X_{i_0}} F_H = P_{i_0} + \sum_{i = 0}^N X_i \partial_{X_{i_0}} P_i
\end{equation}
vanishes on $Z$. 

As previously observed, the polynomial $P_{i_0}$ always vanishes on $Z$.

Let $i$ be some integer such that $0 \leq i \leq N$. Let us show that the term indexed by $i$ in the right-hand side of \eqref{decom partial Z} vanishes on $Z$. 

If $i < v$, then by definition of $Z$, $X_i$ vanishes on $Z$, and so does $X_i \partial_{X_{i_0}} P_i$. 

If $i \geq v$ and $i_0 < v$, then $i_0 < i$, and by definition of the $(P_j)_j$, the polynomial $P_i$ does not depend on the indeterminate $X_{i_0}$, therefore the partial derivative $\partial_{X_{i_0}} P_i$ vanishes on $\PP^N_k$, and so does $X_i \partial_{X_{i_0}} P_i$. 

Finally, if $i \geq v$ and $i_0 \geq v$, then $i \geq u,$ and using Lemma \ref{alpha sing pol depends}, every monomial with nonzero coefficient in~$P_i$ is divisible by some $X_j$ where $0\leq j < v$. Since $i_0 \geq v$, this remains true after taking the partial derivative $\partial_{X_{i_0}}$, so by definition of $Z$, every monomial with nonzero coefficient in $\partial_{X_{i_0}} P_i$ vanishes on $Z$. Therefore the product $X_i \partial_{X_{i_0}} P_i$ vanishes on $Z$. 

So every term in the right-hand side of \eqref{decom partial Z} vanishes on $Z$, and therefore the partial derivative $\partial_{X_{i_0}} F$ vanishes on $Z$.

As this is true for every integer $i_0 \in \{0, \dots, N\}$, the subscheme $Z$ is contained in the singular subscheme $H_{\mathrm{sing}}$, hence by definition, it is contained in $H_{\mathrm{sing}} \cap (X_0 = \dots = X_{v-1} = 0)$.
\end{proof}

The subscheme $Z$  defined by $u+v$ equations of degree $1$ or $d-1 \geq 1$ 
in the projective space~$\PP^N_k$ is of dimension at least $N - (u+v) = s$. Therefore, according to Lemma \ref{Z sing}, if $u \leq v,$ the subscheme $H_{\mathrm{sing}} \cap (X_0 = \dots = X_{v-1} = 0)$ also has dimension at least $s$, which shows inequality \eqref{ineq dim sing}.

This concludes the proof when $u \leq v$. If $u > v$, we define $u' := v$, $v' := u$ so that $u' \leq v'$, $u'+v' = N-s$, and that using inequality \eqref{ineq cond alpha sing} for $u$ and $v$, the following inequality of integers holds:
$$
\deg_{\alpha} F_H  < \alpha_u + (d-1) \alpha_v = \alpha_{v'} + (d-1) \alpha_{u'} \leq \alpha_{u'} + (d-1) \alpha_{v'},
$$
where we have used the inequality $\alpha_{u'} \leq \alpha_{v'}$ and the fact that $d-1$ is at least $1$. Therefore inequality \eqref{ineq cond alpha sing} holds for $u'$ and $v'$, and the triple $(u',v', s)$ satisfies the hypotheses of the Proposition, so using the previous case, we obtain that inequality \eqref{ineq dim sing} holds for $(u',v')$, namely:
\begin{align*}
s &\leq \dim\big(H_{\mathrm{sing}} \cap (X_0 = \dots = X_{v'-1} = 0)\big) 
= \dim\big(H_{\mathrm{sing}} \cap (X_0 = \dots = X_{u-1} = 0) \big) \\
&\leq \dim\big(H_{\mathrm{sing}} \cap (X_0 = \dots = X_{v-1} = 0) \big),
\end{align*}
which shows inequality \eqref{ineq dim sing} in the case where $u > v$.
\end{proof}

\subsection{} We now spell out the corollaries of Benoist's lower bound \eqref{ineq dim sing} which we shall use for the proof of Theorem \ref{MainThm}. 

 To state them, we denote by $s$ the dimension of the singular locus of $H$:
$$s := \dim(H_{\mathrm{sing}}) \in \{-1,\dots,N-1\}.$$
It is equal to $-1$ if and only if $H$ is smooth.

\begin{corollary}
\label{alpha isol}
 Let $g$ be an element of $\GL_{N+1}(k)$, $\alpha = (\alpha_0,\dots,\alpha_N)$ a multi-index in $\Z^{N+1} \setminus \{0\}$ such that $\alpha_0 \leq \dots \leq \alpha_N$ and $\sum_{i = 0}^N \alpha_i = 0$, and $u,v$ two non-negative integers such that $u+v = N-s-1$. 
 The following inequality of integers holds:
\begin{equation}
\label{ineq alpha isol}
\deg_{\alpha}(F_H \circ g) \geq \alpha_u + (d-1) \alpha_v.
\end{equation}
\end{corollary}

\begin{proof}
This is a direct consequence of Proposition \ref{alpha sing} applied to $s' = s+1 \in \{0,\dots, N\}$.
\end{proof}

\begin{corollary} 
\label{alpha isol 2} 
Let $g$ be an element of $\GL_{N+1}(k)$ and $\alpha = (\alpha_0,\dots,\alpha_N)$ a multi-index in $\Z^{N+1} \setminus \{0\}$ such that $\alpha_0 \leq \dots \leq \alpha_N$ and $\sum_{i = 0}^N \alpha_i = 0$. %
When $s \leq N-2$, the following inequality of integers holds:
\begin{equation}
\frac{ N-s-2}{d} \deg_{\alpha}(F_H \circ g) 
 \geq \sum_{i = 1}^{N-s-2} \alpha_i.
\end{equation}
\end{corollary}

\begin{proof}
Using Corollary \ref{alpha isol}, inequality \eqref{ineq alpha isol} holds for every pair $(u,v)$ of non-negative integers such that $u+v = N-s-1$, namely:
$$
\deg_{\alpha}(F_H \circ g) \geq \alpha_u + (d-1) \alpha_v.
$$
Taking the sum of these inequalities for $1 \leq u \leq N-s-2$, and $v := N-s-1-u$, the following inequalities hold:
\begin{align*}
\sum_{u = 1}^{N-s-2} \deg_{\alpha}(F_H \circ g) &\geq \sum_{\substack{1 \leq u \leq N-s-2 \\ v = N-s-1-u}} (\alpha_u + (d-1) \alpha_v)\\
&\geq \sum_{u = 1}^{N-s-2} \alpha_u + (d-1) \sum_{v = 1}^{N-s-2} \alpha_v\\
&\geq d \sum_{u = 1}^{N-s-2} \alpha_u.
\end{align*}
Dividing by $d > 1$ shows the result.
\end{proof}

\section{Proof of Theorem \ref{MainThm} and variants}\label{PfMTvar}

\subsection{}  We are now in position to complete the proof of Theorem \ref{MainThm}.

\subsubsection{} Let $g$ be an element of $\GL_{N+1}(k)$ and $(\alpha_0,\dots,\alpha_N)$ a multi-index in $\Z^{N+1} \setminus \{0\}$ such that $\alpha_0 \leq \dots \leq \alpha_N$ and $\sum_{i = 0}^N \alpha_i = 0$. According to Proposition \ref{degree and HM}, to establish the semistability (resp. the stability) of $H$, we have to show that, under the appropriate hypotheses, the $\alpha$-degree $\deg_{\alpha}(F_H \circ g)$ is non-negative (resp. positive).

Let us denote by $P$ the point $g([0 : \dots : 0 : 1])$ in $\PP^N(k)$. First observe that, if $P$ is not a point in the hypersurface $H$, then according to Proposition \ref{alpha deg loc} \eqref{alpha deg loc non H}, the $\alpha$-degree $\deg_{\alpha}(F_H \circ g)$ is positive.

In the remainder of the proof, we shall always assume that $P$ is contained in the hypersurface $H$, and denote by $\delta_P$ the multiplicity of $H$ at $P$.

If as before we denote by $s$  the dimension of the singular locus of $H$, and if the following inequality of integers holds:
$$
\sum_{i = 1}^{N-s-2} \alpha_i > 0,
$$
then $N-s-2 > 0$, and according to Corollary \ref{alpha isol 2}, the $\alpha$-degree $\deg_{\alpha}(F_H \circ g)$ is positive. 

From now on, we shall  assume that the following inequality holds:
\begin{equation}
\label{sum alpha negative}
\sum_{i = 1}^{N-s-2} \alpha_i \leq 0.
\end{equation}

\subsubsection{} Let us prove Part \eqref{HM isol mult} of Theorem \ref{MainThm}.

Using Proposition \ref{alpha deg loc} \eqref{alpha deg loc sing mult}, the following inequalities hold:
\begin{align}
\deg_{\alpha}(F_H \circ g) &\geq (d - 2 \delta_P) \alpha_N - \delta_P \sum_{i = 1}^{N-1} \alpha_i \nonumber \\ 
&\geq (d - 2 \delta_P) \alpha_N - \delta_P \sum_{i = 1}^{N-s-2} \alpha_i - \delta_P \sum_{i = \max(1, N-s-1)}^{N-1} \alpha_i \nonumber \\
\label{ineq HM for sing point mult isol 1} &\geq (d-2 \delta_P) \alpha_N - \delta_P \sum_{i = \max(1, N-s-1)}^{N-1} \alpha_i \\
\label{ineq HM for sing point mult isol 2} &\geq \Big (d-2 \delta_P - \big (N - \max(1, N-s-1) \big ) \delta_P \Big ) \alpha_N \\
\label{ineq HM for sing point mult isol 3} &\geq \big (d - \delta_P \min(N+1,s+3) \big ) \alpha_N,
\end{align}
where 
in \eqref{ineq HM for sing point mult isol 1} we applied inequality \eqref{sum alpha negative}, and in \eqref{ineq HM for sing point mult isol 2}, we applied the inequality $\alpha_i \leq \alpha_N$.

Consequently, if $d \geq \delta \min(N+1,s+3) \geq \delta_P \min(N+1,s+3)$ (resp. $d > \delta \min(N+1,s+3) \geq \delta_P \min(N+1, s+3)$), then the $\alpha$-degree $\deg_{\alpha}(F_H \circ g)$ is non-negative (resp. positive). 

This completes the proof of Part \eqref{HM isol mult}.

\subsubsection{} For the proof of Part \eqref{HM isol not cone} of Theorem \ref{MainThm}, let us assume that $N \geq 2$ and the projective tangent cone to~$H$ at any 
point where $H$ has multiplicity  $ \delta$ is not the cone over some hypersurface in a projective hyperplane of $\PP_k^{N-1}$. Recall that this may  hold only when~$\delta \geq 2,$ or equivalently when~$s \geq 0$.  

When the multiplicity $\delta_P$ of $H$ at $P$ satisfies $\delta_P < \delta$, inequality \eqref{ineq HM for sing point mult isol 3} still holds and shows that if $d \geq (\delta-1) \min(N+1,s+3) \geq \delta_P \min(N+1,s+3)$ (resp. $d > (\delta-1) \min(N+1,s+3) \geq \delta_P \min(N+1,s+3)$), then the $\alpha$-degree $\deg_{\alpha}(F_H \circ g)$ is non-negative (resp. positive). 

Now let us assume that the multiplicity of $H$ at $P$ is precisely $\delta$. By hypothesis, the projective tangent cone to $H$ at $P$ is not the cone over some hypersurface in a projective hyperplane of $\PP_k^{N-1}$, so using Proposition \ref{alpha deg loc}  \eqref{alpha deg loc sing non cone}, the following inequalities hold:
\begin{align}
\deg_{\alpha}(F_H \circ g) &\geq (d - 2 \delta + 1) \alpha_N - (\delta - 1) \sum_{i = 1}^{N-2} \alpha_i  - (\delta-2) \alpha_{N-1}\nonumber \\
\label{ineq HM for sing point not cone 1} &\geq (d - 2 \delta + 1) \alpha_N - (\delta - 1) \sum_{i = \max(1, N-s-1)}^{N-2} \alpha_i - (\delta-2) \alpha_{N-1} \\
\label{ineq HM for sing point not cone 2} &\geq \Big (d - 2 \delta + 1 - (\delta-1) \big (N-1-\max(1, N-s-1) \big) - (\delta-2)  \Big ) \alpha_N  \\
&\geq \big (d - (\delta-1) \min(N+1,s+3) \big ) \alpha_N, \nonumber
\end{align}
where as before, in \eqref{ineq HM for sing point not cone 1} we applied inequality \eqref{sum alpha negative}, and in \eqref{ineq HM for sing point not cone 2} we applied the inequality $\alpha_i \leq \alpha_N$.

Consequently, if $d \geq (\delta-1) \min(N+1,s+3)$ (resp. $d > (\delta-1) \min(N+1,s+3)$), then the $\alpha$-degree $\deg_{\alpha}(F_H \circ g)$ is non-negative (resp. positive). 

This concludes the proof of  Part \eqref{HM isol not cone}.

\subsection{Variants of Theorem \ref{MainThm}}\label{subsecVar}

The methods used to prove Theorem \ref{MainThm} are somewhat flexible. In this subsection, we discuss two variants of this theorem that can be established with minor changes to the proof.

We adopt the notation of Theorem \ref{MainThm}.

\subsubsection{}\label{421} 

It follows from the use of Benoist's lower bound \eqref{ineq dim sing} in the proof of  Corollary \ref{alpha isol 2} --- and therefore in the proof of Theorem \ref{MainThm} ---  that, in the statement of this theorem, the integer $s$ may be replaced by the maximal dimension $s'$ of the intersection of $H_{\mathrm{sing}}$ with any hyperplane in $\PP^N_k$:
$$s' := \max \big\{ \dim(H_{\mathrm{sing}} \cap V); V  \mbox{ hyperplane in } \PP^N_k\big\}.$$

Observe that $s'$ is equal to $s-1$, unless some $s$-dimensional component of $H_{\mathrm{sing}}$ is contained in some hyperplane, in which case it is $s$. 

Accordingly, when no $s$-dimensional component of $H_{\mathrm{sing}}$ is contained in some hyperplane, we get a stronger version
of Theorem \ref{MainThm}, where for instance Part (1) involves the weaker numerical condition:
$$d \geq \delta \, \min(N+1,s+2).$$
This variant might be useful when dealing with reduced hypersurfaces with non-isolated singularities.

\subsubsection{}

In the proof of Theorem \ref{MainThm}, the application of Part \eqref{alpha deg loc sing non cone} of Proposition \ref{alpha deg loc} may be replaced by an application of Part \eqref{alpha deg loc sing non plane} of the same proposition. This leads to the following result.

\begin{theorem} 
\label{th HM mult}
If for every point $P \in H(k)$ where $H$ has multiplicity $\delta$, the support of the projective tangent cone $\PP(C_P H)$ in $\PP(T_P \PP^N_k) \simeq \PP^{N-1}_k$ is not a projective hyperplane, and if the following condition holds:
\begin{equation}\label{eq th HM mult}
d \geq (N+1) (\delta - 1) \quad (\mbox{resp. } d > (N+1) (\delta - 1)),
\end{equation}
then $H$ is semistable (resp. stable).
\end{theorem}

Unlike the numerical conditions \eqref{eq HM isol mult} and \eqref{eq HM isol not cone} appearing in Theorem \ref{MainThm}, the condition \eqref{eq th HM mult} in Theorem \ref{th HM mult} is not more general when the dimension $s$ of the singular locus of $H$ is small, but Theorem \ref{th HM mult} involves a somewhat more general condition than Theorem \ref{MainThm} on the geometry of the singularities of maximal multiplicity $\delta$.

The criterion in Theorem \ref{th HM mult} applies only when $\delta \geq 2$, and provides a proof of the (semi)stability of some singular hypersurfaces of large degree that relies only on our preliminary estimates in Proposition \ref{alpha deg loc}.

\section{Appendix: Lee's criterion of semistability in terms of  log canonical threshold}\label{LeeLC}

\subsection{Lee's local reinterpretation of the Hilbert-Mumford criterion}
\label{HM weight Lee}

An approach, introduced by Lee \cite{Lee08}  in characteristic zero, then extended by Okawa \cite{Okawa11} to positive characteristic, consists in a reinterpretation of the Hilbert-Mumford criterion in terms of the singularity of pairs attached to the local germs of the hypersurface $H$ in $\PP^N_K$; see \cite{Kollar97B}. 

Let us recall Lee's formulation of the Hilbert-Mumford criterion, changing slightly the numerology to make it compatible with the one in \cite{Benoist14}.

If 
$$
f(x_0,\dots,x_{N-1}) = \sum_{(m_0,\dots,m_{N-1}) \in \N^N} \lambda_{m_0,\dots,m_{N-1}} x_0^{m_0} \dots x_{N-1}^{m_{N-1}}
$$
is a nonzero formal series in $N$ indeterminates and if 
$$w = (w_0,\dots,w_{N-1})$$
is a multi-index of $N$ positive integral (or rational) weights, 
we define the weighted multiplicity of the series $f$ to be the integer:
$$
\mathrm{mult}_w(f) := \mathrm{min}_{\substack{(m_0,\dots,m_{N-1}) \in \N^N \\ \lambda_{m_0,\dots,m_{N-1}} \neq 0}} \sum_{i = 0}^{N-1} m_i w_i.
$$
This integer is $0$ precisely when $f$ does not vanish at the origin.

Let $P$ be a point in  $H(k)$. We define a non-negative real number $I_P(\PP^N_k, H)$ associated to the germ of $H$ at $P$ by:
$$
I_P(\PP^N_k, H) := \mathrm{inf}_{g, w} \frac{\sum_{i = 0}^{N- 1} w_i}{\mathrm{mult}_w\big((F_H \circ g)(X_0,\dots,X_{N-1},1)\big)},
$$
where the infimum is over the element $g \in \GL_{N+1}(k)$ such that $P = g([0 : \dots : 0 : 1])$ and over the multi-index $w$ of positive integral weights.

A simple continuity argument shows that this infimum has the same value as the infimum over every multi-index $w$ of non-negative (integral or rational) weights such that the weighted multiplicity:
$$\mathrm{mult}_w\big((F_H \circ g)(X_0,\dots,X_{N-1},1)\big)$$
is positive.

We can now define the non-negative real number:
$$I(\PP_k^N, H) := \inf_{P \in H(k)} I_P(\PP^N_k, H).$$

\begin{proposition}[{\cite[Lemma 2.1]{Lee08} in characteristic $0$, see also for instance \cite[Lemma 2.7]{Okawa11} in positive characteristic}]
\label{HM weights}
The hypersurface $H$ is semistable (resp. stable) if and only if $I(\PP^N_k, H) \geq \frac{N+1}{d}$ (resp. $I(\PP^N_k, H) > \frac{N+1}{d}$).
\end{proposition}

The link between Proposition \ref{HM weights} and the Hilbert-Mumford criterion as given in Proposition \ref{degree and HM} is given by the following equality, where $g$ is an element of $\GL_{N+1}(k)$, where $\alpha = (\alpha_0,\dots,\alpha_N)$ is a multi-index in $\Z^{N+1}$ such that $\vert \alpha \vert = 0$ and $\alpha_0 \leq \dots \leq \alpha_N$, and where $w = (w_0,\dots,w_{N-1})$ is the multi-index in $\Z_{\geq 0}^N$ given by, for every $i$, $w_i = \alpha_N - \alpha_i$:
$$\deg_\alpha(F_H \circ g) = \frac{d}{N+1} \sum_{i = 0}^{N-1} w_i - \mathrm{mult}_w\big((F_H \circ g)(X_0,\dots,X_{N-1},1)\big).$$

\subsection{Log canonical and log terminal pairs}

Let us recall the definition and elementary properties of log canonical pairs.\footnote{See \cite{Kollar97B}, mainly \S3, and \cite[2.3]{KollarMori98}, that we closely follow,  for more details.}

Let $(X,D)$ be a pair where $X$ is a normal integral algebraic $k$-scheme, and where~$D= \sum_i a_i D_i$ is a Weil $\Q$-divisor, written as a sum of distinct prime divisors, with coefficients $a_i$ in $\Q$. 

We denote by $K_X$ a Weil divisor in $X$ that extends the Cartier divisor in the regular locus $X_{\mathrm{reg}}$ of $X$ attached to a nonzero rational section of the canonical line bundle $\omega_{X_{\mathrm{reg}}}$, and we assume that~$m(K_X + D)$ is Cartier for some positive integer~$m$.

Consider a proper birational morphism:
\[
\nu : X' \lra X
\]
where $X'$ is a normal integral algebraic $k$-scheme.  Let us denote by $\mathrm{Ex}$ the exceptional locus of~$\nu$ in~$X'$, and by~$E_j$ the integral exceptional divisors, namely the components of  $\mathrm{Ex}$ of codimension~$1$. Consider also the strict transform of $D$ in $X'$, namely the Weil $\Q$-divisor in $X'$:
$$\nu_\ast^{-1} D := \sum_i a_i \, \nu_\ast^{-1} D_i.$$

The two line bundles:
$$\cO_{X'}\big(m(K_{X'} + \nu_\ast^{-1} D)\big)_{\mid X' \setminus \mathrm{Ex}} 
\quad \mbox{and} \quad \nu^\ast \cO_X \big(m(K_X + D)\big)_{\mid X' \setminus \mathrm{Ex}}$$
are naturally isomorphic. Therefore there exist rational numbers $a(E_j,  X, D)$ in $(1/m) \Z$ such that this isomorphism extends to an isomorphism:
$$\cO_{X'}\big(m(K_{X'} + \nu_\ast^{-1} D)\big) \simeq \nu^\ast \cO_X \big(m(K_X + D)\big)
\Big( \sum_j m \, a(E_j, X, D) \, E_j \Big). $$

If $E$ is any prime divisor in $X'$, we define its \emph{discrepancy} $a(E,X,D)$ to be $a(E_j, X, D)$ when $E$ is $E_j$, to be $-a_i$ when $E$ is the strict transform $\nu_\ast^{-1} D_i$ of $D_i$, and to be $0$ in other cases.

The  discrepancy $a(E,X,D)$ only depends on the local ring of $X'$ at the generic point of $E$, seen as a discrete valuation ring in the fraction field of $X$. In other words, if $X'' \rightarrow X'$ is another proper birational morphism, where $X''$ is an integral normal algebraic $k$-scheme, 
then the discrepancy of the strict transform of~$E$ in~$X''$ is the same as the discrepancy of~$E$.

The pair $(X,D)$ is called \emph{log canonical} (resp. \emph{log terminal}) if for every proper birational morphism $X' \rightarrow X$ with $X'$ a normal integral algebraic $k$-scheme 
and for every integral exceptional  divisor $E$ in~$X'$, the following inequality holds:
\begin{equation}\label{discrep lc}
a(E,X,D) \geq -1 \quad (\text{resp.} > -1).
\end{equation}

If one of the components $D_i$ of $D$ has multiplicity $a_i > 1$, then  the pair $(X,D)$ is not log canonical.  Indeed, blowing up an integral closed subscheme of codimension 2 contained in $D_i$, then blowing up the intersection of the exceptional divisor and the strict transform of $D_i$, and then repeating this process, one constructs a sequence of exceptional divisors $(E_r)_{r > 0}$ with discrepancy $r (1 - a_i)$; this diverges toward $-\infty$ as $r$ goes to~$+\infty$.

As the discrepancy of a strict transform of a component of $D$ is the opposite of its multiplicity, the previous observation shows that the above definition of log canonicity is equivalent to the one  where the inequality $a(E,X,D) \geq -1$  is required to hold  for every integral divisor~$E$ in $X'$.

\subsection{The log canonical threshold of a $\mathbb{Q}$-divisor at a point}

Let us now recall the definition of the log canonical threshold of a $\mathbb{Q}$-divisor.\footnote{ See \cite{Kollar97B}, mainly \S8.}

With the notation of the previous subsection, the \emph{log canonical threshold} of~$D$ at a point~$P$ in~$\vert D \vert$ is defined as the supremum:
\[
\text{lct}_P(X,D) := \sup\big \{c \geq 0 \; | \; (X, c.D) \text{ is log canonical in a neighborhood of } P\big\}.
\]

If $D$ is an effective Cartier divisor, then letting $c > 1$ be a real number and $D_i$ be a component of $D$ with multiplicity $a_i \geq 1$, the following inequality holds:
\[
a(D_i, X, c . D) = - c \, a_i < - a_i \leq -1.
\]
Consequently, if $D$ is an effective Cartier divisor, then the following inequality holds:
$$\mathrm{lct}_P(X,D) \leq 1.$$

Moreover, we have the following upper bound for the log canonical threshold of an effective Cartier divisor in $\A^N_k$:

\begin{proposition}[{\cite[Proposition 8.13]{Kollar97B} in characteristic zero, \cite[Proposition 4.9]{Okawa11} in positive characteristic}] 
\label{weight log canon} 
Let $N$ be a positive integer, $f$ a nonzero regular function on the affine space~$\mathbb{A}^N_k$, and let $D$ be the Cartier divisor in $\mathbb{A}^N_k$ defined by the vanishing of~$f$. Let $w = (w_0, \ldots,w_{N-1})$ be a multi-index of $N$ positive integral weights. The following inequality holds:
$$
\lct_0(\mathbb{A}^N_k, D) \leq \min\left(1, \frac{\sum_i w_i}{\mathrm{mult}_w(f)}\right),
$$
where $\mathrm{mult}_w(f)$ is the weighted multiplicity of the function $f$ introduced in Subsection \ref{HM weight Lee}.
\end{proposition}
 
Combining this result with Proposition \ref{HM weights}, Lee obtains the following result.

\begin{proposition}[{\cite[Proposition 2.5]{Lee08} in characteristic zero, \cite[Theorem 4.1]{Okawa11} in positive characteristic}] 
\label{log canon semistab}
If $H$ is a hypersurface of degree $d$ in the projective space $\mathbb{P}^N_k$, and if for every point $P$ in $H(k)$, the following inequality holds:
\begin{equation}
\label{ineq log canon semistab}
\lct_P(\mathbb{P}^N_k, H) \geq \frac{N+1}{d} \quad \Big(\text{resp. } > \frac{N+1}{d}\Big),
\end{equation}
then $H$ is semistable (resp. stable).

In particular, if the pair $(\PP^N_k, H)$ is log canonical and if $d \geq N+1$ (resp. $>N+1$), then $H$ is semistable (resp. stable).
\end{proposition} 

Observe that since the log canonical threshold of a non-empty effective Cartier divisor is at most~1, inequality \eqref{ineq log canon semistab} cannot hold unless $d \geq N+1$ ($d > N+1$ for the strict variant).

The following result is useful for applications of this criterion.

\begin{proposition}[{\cite[Corollary 3.13]{Kollar97B}}]
\label{log canon on resol}
Let $D$ be a Weil $\mathbb{Q}$-divisor in a normal integral algebraic $k$-scheme~$X$ such that $K_X + D$ is $\mathbb{Q}$-Cartier. Let us consider a \emph{log resolution} of the pair $(X,D)$, namely a proper birational morphism
$$
\nu : X' \lra X
$$
with~$X'$ a connected smooth $k$-scheme such that 
the following divisor in~$X'$:
$$\nu^{-1}(D) \cup \bigcup_{\substack{E \text{ exceptional}\\ \text{ divisor in } X'}} E$$
is a divisor with simple normal crossings. The pair $(X,D)$ is log canonical if and only if the discrepancy of every  divisor in~$X'$ is~$\geq -1$.

\end{proposition}

With the notation of Proposition \ref{log canon on resol}, if the pair $(X,D)$ admits a log resolution on a neighborhood of $P$ --- which is the case  when the base field $k$ has characteristic zero 
--- then the log canonicity of a pair $(X,c . D)$ near $P$ can be checked on the exceptional divisors of such a log resolution and the multiplicities of the components of $D$. In particular, the supremum defining the log canonical threshold is actually a maximum in this case, and can be computed using the geometry of this resolution and the preimage of~$D$.

Consequently, in order to apply Proposition \ref{log canon semistab}, it is enough to know the geometry of a log resolution of the pair $(\mathbb{P}^N_k, H)$ at every singular point of $H$. However it is not sufficient to know the multiplicities of $H$ at these singularities, as their  knowledge does not give a lower bound on the log canonical threshold.

Let us also mention the following result by Tian. Its statement is similar to Proposition \ref{log canon semistab}, but its proof relies on completely different, analytic arguments involving the $K$-energy.

\begin{proposition}[{\cite[Theorem 0.2]{Tian94}}]\label{log term semistable}
\label{log term stab}
If the characteristic of the field $k$ is $0$, if the pair $(\mathbb{P}^N_k,H)$  is log terminal, and if the following inequality holds:
\[
d \geq N+1,
\]
then $H$ is stable.
\end{proposition}

\subsection{Comparison of Theorem \ref{MainThm} with Lee's results}

The criterions of (semi)stability given in Propositions \ref{log canon semistab} and \ref{log term semistable} and in  our Theorem \ref{MainThm} and its variants in Subsection \ref{subsecVar} are somewhat similar: all of them involve a numerical condition, in the form of a lower bound on the degree of the hypersurface, compared to invariants depending on the local geometry of its singularities.%

An obvious advantage of our results is that, contrary to the ones of Lee and Tian, they do not require the assumption:
$$d \geq N+ 1.$$

We may also compare the numerical criteria for (semi)stability provided by Proposition \ref{log canon semistab} and Theorem \ref{MainThm} in some cases where both apply and where these invariants of singularities are easily computed.
    We shall consider two specific cases:  generic images  of smooth surfaces into $\PP^3_k$, and hypersurfaces with isolated singularities admitting a smooth projective tangent cone (that is, isolated semi-homogeneous singularities).

\subsubsection{} Let $S$ be a smooth surface in some projective space $\PP^r_k$. For every projective subspace $L$ of~$\PP^r_k$ of codimension $4$ not containing $S$, the projection of center $L$:
$$\PP^{r}_k \setminus L \twoheadrightarrow \PP^3_k,$$
induces a dominant rational map between $S$ and some reduced hypersurface $H_L$ in $\PP^3_k$.

According to \cite[Theorem 3]{Roberts71}\footnote{In positive characteristic, the embedding $S \hookrightarrow \PP^r_k$ may have to be replaced by its left composition with a Veronese embedding $\PP^r_k \hookrightarrow \PP^{r'}_k$ of degree at least 2 for Roberts' result to hold. Over the complex numbers, this genericity result goes back to classical geometers of the 19th century, and in some form was already known  to Cayley.}, if $L$ is picked in some dense open subset of the Grassmannian scheme $\mathrm{Grass}_{r-4}(\PP^r_k)$, then this rational map is a finite birational morphism, and the hypersurface $H_L$ in $\PP^3_k$ only has normal crossing singularities and so-called Whitney umbrella singularities, namely singularities near which, in some local system of coordinates $(x,y,z)$, the surface $H_L$ has an equation of the form:
$$x^2 - y^2 z = 0.$$

Let us fix a projective subspace $L$ in such a dense open set, and let us denote by $d$ the degree of the resulting hypersurface $H_L$. Since pairs attached to hypersurfaces with normal crossing singularities or Whitney umbrella singularities are log canonical, for every point $P$ in $H_L(k)$, the following equality holds:
$$\mathrm{lct}_P(\PP^3_k, H_L) = 1.$$

Therefore Lee's numerical condition for semistability of the projective hypersurface $H_L$, namely:
$$\mathrm{lct}_P(\PP^3_k, H_L) \geq  \frac{4}{d} \quad \mbox{for every $P \in H_L(k)$},$$ 
becomes the following inequality:
\begin{equation} 
\label{cond Lee surf embed}
d \geq 4.
\end{equation}

To apply Theorem \ref{MainThm} to $H_L$, observe that normal crossing singularities given by the local intersection of two branches and Whitney umbrella singularities are of multiplicity $2$, with projective tangent cones in~$\PP^2_k$ given respectively by the union of two distinct hyperplanes and by one hyperplane with multiplicity~$2$. These cones are cones over hypersurfaces in projective hyperplanes~$\PP^1_k$. Moreover, normal crossing singularities given by the local intersection of three branches are of multiplicity~$3$, with projective tangent cones given by the union of three hyperplanes in general position in~$\PP^2_k$. These projective tangent cones are not cones over hypersurfaces in projective hyperplanes~$\PP^1_k$.

Therefore we can apply Theorem \ref{MainThm} with $s \in \{0,1\}$ to $H_L$: if the hypersurface $H_L$ admits normal crossing singularities given by the local intersection of three branches, we apply part (2) with $\delta = 3$, if there is no such singularity, we apply part (1) with~$\delta = 2$. 

In both cases, the numerical condition for semistability of the projective hypersurface $H_L$ is the following inequality:
\begin{equation} 
\label{cond Mordant surf embed}
d \geq 8.
\end{equation}

When moreover no irreducible component of the double curve of $H_L$ is  contained in some projective plane, the variant of Theorem \ref{MainThm} discussed in \ref{421} above shows that $H_L$ is semistable  when:
\begin{equation}
\label{cond Mordant surf embed var}
d \geq 6.
\end{equation}

Conditions \eqref{cond Mordant surf embed} and \eqref{cond Mordant surf embed var} are less general than the condition \eqref{cond Lee surf embed}.

\subsubsection{} If a Cartier divisor $D$ in a smooth $N$-dimensional $k$-scheme $X$ has an isolated singularity with multiplicity $\delta_P$ at a point $P$ such that the projective tangent cone $\PP(C_P D)$ is a smooth hypersurface in the projective space $\PP(T_P X) \simeq \PP^{N-1}_k$, then the blow-up of $X$ at $P$ is a log resolution of~$(X,D)$ in a neighborhood of~$P$. Consequently the log canonical threshold of such a singularity is given by:
$$\mathrm{lct}_P(X,D) = \min(1, N/\delta_P).$$ 

Let us consider a singular projective hypersurface $H$ in $\PP^N_k$, such that every singularity of $H$ is as above, and let us denote by $\delta \geq 2$ the maximal multiplicity of $H$ at a point of $H(k)$. 
    
Lee's numerical condition given by Proposition \ref{log canon semistab} for such a hypersurface to be semistable is the following one:
\begin{equation}      
\label{cond Lee semi hom} \min(1, N / \delta) 
\geq \frac{N+1}{d} \quad \Longleftrightarrow \quad \Big[ d \geq N+1 \quad \mbox{and} \quad d \geq \frac{(N+1) \delta}{N}\Big].
\end{equation}

Moreover every singularity $P$ of $H$ is isolated, and its projective tangent cone is a smooth hypersurface in $\PP(T_P \PP^N_k)$; in particular, if $N \geq 2$, it is not the cone of a hypersurface in some projective hyperplane of $\PP(T_P \PP^N_k)$. Consequently, if $N \geq 2$. the geometrical conditions of Theorem \ref{MainThm}, Part (2) hold in this case with $s = 0$, and the numerical condition \eqref{eq HM isol not cone} becomes the inequality:
\begin{equation}
\label{cond Mordant semi-hom} d \geq 3 (\delta - 1).
\end{equation}

Observe that, of the two numerical conditions \eqref{cond Lee semi hom} and \eqref{cond Mordant semi-hom} (the latter when $N \geq 2$), neither is more general than the other.

%
%

\end{document}